\documentclass[11pt]{article}
\usepackage[utf8]{inputenc}
\usepackage[numbers]{natbib}
\title{Omega-categorical limits of betweenness relations and $D$-sets}
\author{ Asma Ibrahim Almazaydeh$^{1}$, Samuel Braunfeld$^{2}$, and Dugald Macpherson$^{3}$\\
$^1$ Department of Mathematics, Tafila Technical University, Tafila, Jordan\\
$^{2}$ Computer Science Institute, Charles University, 11800 Prague, Czech Republic\\ Computer Science Institute, Czech Academy of Sciences, 18200 Prague, Czech Republic\\
$^3$ Department of Mathematics,  University of Leeds, Leeds, United Kingdom}

\usepackage{pgf}

\usepackage{amssymb,amsbsy,times,fancyhdr,color}
\usepackage{amsmath}
\usepackage[dvips]{epsfig}
\usepackage{graphicx}
\graphicspath{ {images/} }
\usepackage{latexsym}
\usepackage{amsthm}
\usepackage{float,graphicx, enumerate}
\usepackage[margin=3cm]{geometry}
\usepackage[none]{hyphenat}

\usepackage{hyperref}
\hypersetup{pdfborder={0 0 0}}

\usepackage[all]{xy}
\usepackage{amsopn}

\usepackage{color, colortbl}

\usepackage{arabtex}
\usepackage{utf8}

\usepackage{latexsym}

\usepackage{verbatim}
\usepackage{amssymb}

\usepackage{mathrsfs}

\usepackage{amsthm}

\usepackage{inputenc}
\usepackage{tikz}
\usetikzlibrary{trees}
\usepackage{textcomp}
\usepackage{romannum}
\usepackage{wasysym}

\usepackage{color}

\usepackage{lipsum}  
\usepackage{graphicx}
\usepackage{caption}
\usepackage{float}
\usepackage{subcaption}
\usetikzlibrary{decorations.pathreplacing,angles,quotes}
\usepackage{mathtools}
\usepackage{MnSymbol}

\usetikzlibrary{arrows,decorations.markings,plotmarks}

\usepackage [pagewise]{lineno}
\usepackage{graphicx}
\newtheorem{mydef}{Definition}[section]
\newtheorem{thm}[mydef]{Theorem}
\newtheorem{lem}[mydef]{Lemma}
\newtheorem{prop}[mydef]{Proposition}

\newtheorem{ex}[mydef]{Example}

\def\Aut{\mathop{\rm Aut}\nolimits}
\def\acl{\mathop{\rm acl}\nolimits}
\def\Sym{\mathop{\rm Sym}\nolimits}
\def\lev{\mathop{\rm lev}\nolimits}
\def\Min{\mathop{\rm Min}\nolimits}
\def\Max{\mathop{\rm Max}\nolimits}
\def\eq{\mathop{\rm eq}\nolimits}
\def\ram{\mathop{\rm ram}\nolimits}
\def\tp{\mathop{\rm tp}\nolimits}
\begin{document}

\maketitle
\pagenumbering{arabic}

\maketitle
\begin{abstract}
We explore  two constructions of oligomorphic Jordan permutation groups preserving a `limit of betweenness relations' and a `limit of $D$-relations', from \cite{bhattmacph2006jordan} and \cite{almazaydeh2021jordan} respectively. Several issues left open in \cite{almazaydeh2021jordan} are resolved. In particular it is shown that the `limit of $D$-relations' is not homogeneous in the given language, but is `homogenizable', that is, there is a homogeneous structure over a finite relational language with the same universe and the same automorphism group. The structure is NIP, but not monadically NIP, its age is not well-quasi-ordered under embeddability, and the growth rate of the sequence enumerating orbits on $k$-sets grows faster than exponentially. The automorphism group is maximal-closed in the symmetric group. Similar results are shown for the construction in  \cite{bhattmacph2006jordan}.

\end{abstract}
\section{Introduction}
If $G$ is a transitive permutation group on a set $\Omega$, then a subset $\Gamma$ of $\Omega$ is a {\em Jordan set} if $|\Gamma|>1$ and the pointwise stabiliser $G_{(\Omega\setminus \Gamma)}$ is transitive on $\Gamma$. It is a {\em proper} Jordan set if in addition, if $G$ is $(k+1)$-transitive then $|\Omega\setminus \Gamma|\neq k$, that is, the Jordan condition does not arise just from the degree of transitivity. A {\em Jordan group} is a transitive permutation group with a proper Jordan set. Finite Jordan groups which are {\em primitive} (i.e. preserve no proper non-trivial equivalence relation on the set acted on) were classified in  \cite{neumann1985some} -- see also \cite{kantor1985homogeneous} and the Appendix of \cite{cherlin1985aleph}. A structure theorem for {\em infinite} primitive Jordan groups was given in \cite{adeleke1996classification}, building on earlier work of Adeleke and Neumann in \cite{adeleke1996primitive}. This structure theorem states that  if  $(G,\Omega)$ is a primitive  Jordan group which is not {\em highly transitive} (i.e. for some $k\geq 1$ is not $k$-transitive), then $G$  preserves on $\Omega$  a structure of one of the following types: a Steiner system (possibly with infinite `lines'); a linear order, circular order, linear betweenness relation or separation relation; a semilinear order, `general' betweenness relation, $C$-relation or $D$-relation; or a `limit' of Steiner systems, betweenness relations or $D$-relations. The three `limit' constructions are rather different in that no explicit class of invariant relational structure is specified. 

A construction of a group preserving a limit of Steiner systems was given by Adeleke in \cite{adeleke1995semilinear}, and developed further by Johnson in \cite{johnson2002constructions}. Adeleke also in \cite{adeleke2013irregular} gave examples of groups preserving a limit of betweenness relations and a limit of $D$-relations. His constructions are not oligomorphic, at least in the case of a limit of betweenness relations, and no invariant relational structure is explicit. The paper \cite{bhattmacph2006jordan}, which builds on early drafts of \cite{adeleke2013irregular}, constructs an $\omega$-categorical structure $\mathcal{M}_B$ whose automorphism group preserves a limit of betweenness relations, and the same is done (with an analogous structure $\mathcal{M}_D$) for limits of $D$-relations in \cite{almazaydeh2021jordan}. These constructions are intricate, but we believe them to be essentially new treelike constructions which may have significance for other reasons. A number of fine structural issues were left open in \cite{bhattmacph2006jordan}
and \cite{almazaydeh2021jordan}, and we aim here to resolve some of these. We note also the paper \cite{bwt}, which takes a rather more general approach to the construction of limits of betweenness relations of
\cite{bhattmacph2006jordan}, extending beyond the $\omega$-categorical context.

A countably infinite relational structure $M$ is said to be {\em homogeneous} if every isomorphism between finite substructures of $M$ extends to an automorphism. Following Covington \cite{covington1990homogenizable}, we say that $M$ is {\em homogenizable} if it can be made homogeneous by adapting the language, that is, if there is a homogeneous structure $M'$ over a finite relational language such that $M$ and $M'$ have the same universe and the same automorphism group.

{\bf Notation.} Generally, we prefer not to distinguish notationally between a structure and its universe, but for $\mathcal{M}_D$ and $\mathcal{M}_B$ this seems necessary, largely because for $\mathcal{M}_D$ we consider the structure under three different languages, all with the same automorphism group, and similarly for $\mathcal{M}_B$. Our convention is as follows for $\mathcal{M}_D$, and is essentially the same for $\mathcal{M}_B$. When the choice of language is unimportant, and the focus is just on 0-definable sets and orbits, we just write $\mathcal{M}_D$. The structure constructed in \cite{almazaydeh2021jordan}, in a language with relation symbols $L, S, L', S', R,Q$ will be denoted  by $\mathcal{M}_D^1$. Its reduct to the language with just the ternary relation symbol $L$ will be denoted $\mathcal{M}_D^0$. In the language in which it is homogeneous, which has relation symbols $L,S, Q^{\leq}, Q^{\geq}, P,T$, it will be denoted $\mathcal{M}_D^2$. The universe will just be denoted $M_D$. We may write $\mathcal{M}$ for one of these structures, but except with these structures, we do not distinguish notationally between a structure and its domain.

We prove here the following result, collecting several elementary observations about $\mathcal{M}_B$ and $\mathcal{M}_D$. We say that a structure $M$ has {\em trivial algebraic closure} if $\acl(A)=A$ for all $A\subset M$.

\begin{prop} \label{homogenizable} The following hold, for $\mathcal{M}\in \{\mathcal{M}_B^1, \mathcal{M}_D^1\}$.
\begin{enumerate}
    \item $\mathcal{M}$ is not homogeneous, but is homogenizable.
    \item $\mathcal{M}$ is NIP.
    \item $\mathcal{M}$ has trivial algebraic closure.
    \end{enumerate}
  \end{prop} 

We view the family of structures explored in  \cite{adeleke1998relations} -- namely semilinear orders, betweenness relations, and $C$ and $D$ relations -- as treelike structures. They are all built from (lower) semilinear orders: betweenness relations are obtained as reducts of semilinear order by replacing the ordering by an induced betweenness, $C$-relations have as universe a dense set of maximal chains of a semilinear order, and a $D$-relation can be obtained from a $C$-relation by forgetting the downwards direction (and can be viewed as a dense set of ends in a betweenness relation). In this paper we shall refer to such structures as the {\em basic} treelike structures. The main theorem of \cite{adeleke1996classification} says that any  primitive but not highly transitive Jordan group $(G,\Omega)$ with a proper {\em primitive} Jordan set preserves on $\Omega$ either a linear order, linear betweenness relation, circular order, or separation relation, or a basic treelike structure.
There is a uniformity in the method of proof -- essentially, in each case, the relational structure is identified by finding a group-invariant family of subsets of $\Omega$ with certain specified intersection properties.

The structures $\mathcal{M}_B$ and $\mathcal{M}_D$ are entitled to be called treelike, even though their automorphism groups do not preserve any basic treelike structure on the domain. The universe is the union of an invariant family of subsets which is ordered (under inclusion) by an invariant semilinear order; each of the sets in this family has a unique maximal invariant equivalence relation, with an invariant  betweenness relation or $D$-relation (for $\mathcal{M}_B$ or $\mathcal{M}_D$ respectively) on the quotient; and for both $\mathcal{M}=\mathcal{M}_B$ and $\mathcal{M}=\mathcal{M}_D$, if $a\in M$ then there is an $a$-definable $C$-relation on $M\setminus \{a\}$. 

There are several strong model-theoretic and combinatorial conditions which, among finitely homogeneous and more generally $\omega$-categorical structures, appear to coincide, and we now describe these. If $G$ is an oligomorphic group on $\Omega$, let $f_k(G)$ denote the number of orbits of $G$ on the collection $\Omega^{[k]}$ of unordered $k$-subsets of $\Omega$. It was shown in \cite{macpherson1985orbits} that if $(G,\Omega)$ is primitive but not highly homogeneous then $(f_k(G))$ grows at least exponentially, and more detailed results on growth rates, answering several questions from \cite{macpherson1985orbits} and \cite{macphersongrowth2}, are obtained in \cite{braun-growth}. We shall say that $(f_k(G))$ has {\em super-exponential growth} if, for every $c>1$, there is $K\in \mathbb{N}$ such that $f_k(G)>c^k$ for every $k>K$. By \cite{macphersongrowth3},  if $M$ is  $\omega$-categorical and has the independence property, then $f_k(\Aut(M))$ has super-exponential growth; in fact, for any $\epsilon>0$, $f_k(\Aut(M))>2^{k^{1+\epsilon}}$ for  sufficiently large $k$. It is shown in \cite{cameron1987some}, extending earlier results, that there are basic treelike structures of all four kinds with non-super-exponential growth -- in fact, there are examples of interesting homogeneous {\em expansions} of 
such structures with non-super-exponential growth, such as a $C$-relation equipped with a compatible linear order or a $D$-relation with a compatible circular order.

Recall that  the {\em age} Age$(M)$ of a relational structure $M$ is the collection of finite structures which embed in $M$. It is well-known that for certain rather rare homogeneous structures $M$, Age$(M)$ is well quasi-ordered under embeddability -- that is, there are no infinite antichains (and, as is automatic, no infinite descending chains). We shall say that such structures have {\em wqo age}. Using Kruskal's Theorem, one can rather easily show that natural examples of basic treelike structures have wqo age.

{\em Monadically NIP} structures, for which any expansion by unary predicates remains NIP, were introduced by Baldwin and Shelah in \cite{bshelah}, and their study continued in \cite{bl}. There, extending \cite[Conjecture 2.9]{homogsurvey}, Conjecture 1 states that for an $\omega$-categorical homogeneous relational structure $M$, the following conditions are equivalent: 
$M$ is monadically NIP; $(f_k(\Aut(M))$ is bounded above exponentially; and $M$ has wqo age. In this direction, the authors prove the following (\cite[Theorem 1.2]{bl}). Here, the growth rate of Age$(M)$ refer to the growth of the sequence counting isomorphism types of $k$-element structures in the age.
\begin{thm}[{\cite[Theorem 1.2]{bl}}] 
Let $M$ be an $\omega$-categorical structure that is not monadically NIP. Then $(f_k(Aut(M)))$ is asymptotically greater than $\lfloor k/\ell\rfloor!$ for some $\ell\in \mathbb{N}^{>0}$.
\end{thm}

Originally we had hoped that limits of betweenness relations and $D$-sets might have provided a counterexample to \cite[Conjecture 1]{bl}, but the result below shows they do not. 
\begin{thm} \label{growth} Let $\mathcal{M}\in \{\mathcal{M}_B^2, \mathcal{M}_D^2\}$. Then
\begin{enumerate}
    \item $\mathcal{M}$ is not monadically NIP.
    \item $(f_k(\Aut(\mathcal{M}))$ has super-exponential growth rate.
    \item Age$(\mathcal{M})$ is not wqo.
\end{enumerate}
\end{thm}

Recall that if $M$ is countably infinite, then $\Sym(M)$ carries a topological group structure, where the basic open sets are cosets of pointwise stabilisers of finite sets, and that the topology is metrizable, giving $\Sym(M)$  (or the symmetric group on any countably infinite set) the structure of a Polish group. Its closed subgroups are exactly the automorphism groups of first order structures with universe $M$. If $M$ is a countably infinite structure, we say that $\Aut(M)$ is {\em maximal-closed} in $\Sym(M)$ if  for every closed subgroup $H$ of $\Sym(M)$ with $\Aut(M)\leq H\leq \Sym(M)$ we have $H=\Aut(M)$ or $H=\Sym(M)$. By the Ryll-Nardzewski Theorem, if $M$ is $\omega$-categorical then $\Aut(M)$ is maximal-closed in  $\Sym(M)$ if and only if $M$ has no proper non-trivial {\em reduct} in the sense of \cite{thomas}. We prove the following, solving Problem 6.4 of \cite{almazaydeh2021jordan} .
\begin{thm} \label{max}
    Let $\mathcal{M}\in \{\mathcal{M}_B,\mathcal{M}_D\}$. Then $\Aut(\mathcal{M})$ is maximal-closed in  $\Sym(M)$.
    \end{thm}

We comment briefly on the motivation  of this body of work (the present paper, and also \cite{bhattmacph2006jordan} and \cite{almazaydeh2021jordan}). First, we believe that the class of Jordan permutation groups is a natural and important class. Several applications of structural results on Jordan groups are described in the introduction to \cite{almazaydeh2021jordan}. We do not give detail on these here, but they include: the structure of $\omega$-categorical strictly minimal sets (see the classification given in \cite{cherlin1985aleph}) and analogous results on permutation groups in \cite{neumann1985some} (a description of groups acting primitively on a  countably infinite set $X$ with no countable orbits on the collection of infinite co-infinite subsets of $X$); classification results for primitive groups of uncountable degree containing a  non-identity element of small support (see \cite{adeleke1996infinite}); analogues (see \cite{macpherson-praeger-cycle}) for other cycle types of Wielandt's result that any primitive group of finite degree containing a non-identity finitary permutation is highly transitive; several results showing that certain closed infinite permutation group are maximal-closed in the symmetric group (see e.g. \cite{bodirsky-macpherson} for a non-oligomorphic example of countable degree, and \cite{kaplan2016affine} concerning affine and projective groups).

The further motivation to better understand these limit constructions is that they have potential to provide useful counter-examples, in model theory and permutation group theory. In the work in \cite{adeleke1996classification} giving a structure theorem for primitive Jordan groups, their existence came initially as a surprise, not anticipated in the classification results in \cite{adeleke1996primitive} of primitive Jordan groups with a proper {\em primitive} Jordan set. (We add, though, that Adeleke's initial manuscript giving existence of groups preserving limits of betweenness and $D$-relations, which led to \cite{adeleke2013irregular}, was from  the early 1990s so predates \cite{adeleke1996classification}.) 

Several questions on these limit constructions are posed in \cite[Section 6]{almazaydeh2021jordan}. This paper answers Problem 6.4 and sheds light on Problem 6.6, but the other problems remain.  We also mention that a more precise classification of closed {\em oligomorphic} primitive Jordan groups may be feasible. There are further questions on the model theory of $\mathcal{M}_B$ and $\mathcal{M}_D$ -- for example we expect that the analysis of indiscernible sequences in Lemmas~\ref{ind-D} and \ref{ind-B} should yield that they are dp-minimal and distal. We also ask whether the construction of $\mathcal{M}_B$ can be adapted to put $\mathbb{R}$-tree structure on each betweenness relation, and whether there are analogues in which these structures are discrete, possibly with interesting group-theoretic consequences.

We give background on limits of betweenness relations and $D$-relations in Section 2. Proposition~\ref{homogenizable}  and Theorem~\ref{growth} are proved in Section 3, and Theorem~\ref{max}  in Section 4. For the relevant model-theoretic background (in particular, NIP theories, and dp-minimality) \cite{simon2015guide} should suffice as a reference. 

\section{Background on the constructions}
In this section we give an overview of the structures $\mathcal{M}_B$ and $\mathcal{M}_D$. We shall assume some familiarity with the treelike structures examined in detail in \cite{adeleke1998relations}, namely 
semilinear orders, betweenness relations, $C$-relations and $D$-relations. Some background on these is also given in \cite{bhattmacph2006jordan} and \cite{almazaydeh2021jordan}.

First, we recall some standard permutation group terminology. A permutation group $(G,X)$ is {\em primitive} if $G$ preserves no proper non-trivial equivalence relation on $X$. For $k>1$ we say $G$ is {\em $k$-transitive} (respectively, {\em $k$-homogeneous}) if it is transitive on the collection of ordered (respectively, unordered) $k$-subsets of $X$, and that $(G,X)$ is {\em highly transitive} if it is $k$-transitive for all $k$. If $k\geq 2$ we say that $(G,X)$ is {\em $k$-primitive} if it is $(k-1)$-transitive and for distinct $x_1,\ldots, x_{k-1}\in X$, $G$ acts primitively on $X\setminus\{x_1,\ldots,x_{k-1}\}$. 

Next, we give the key definition from \cite{adeleke1996classification}.

    \begin{mydef}(\cite{adeleke1996classification}, Definition 2.1.9) \label{limits} \em If $(G, X)$ is an infinite Jordan group we say that $G$ preserves a {\em limit of $D$-relations} if
there are: a linearly ordered set $(J, \leq)$ with no least element, a  chain $(Y _i :i\in J)$
of subsets of $X$ and  chain $(H_i: i\in J)$ of subgroups of $G$ with $Y_i\supset Y_j$ and $H_i>H_j$ whenever $i<j$, such that the following hold:
\begin{enumerate}[(i)]
    \item for each $i, H_i =G_{(X \backslash Y_i) }$, and $H_i$ is transitive on $Y_i$ and has a unique non-trivial
maximal congruence $\sigma_i$ on $Y_i$;
\item for each $i$, $(H_i, Y_i / \sigma_i)$ is a 2-transitive but not 3-transitive Jordan group preserving a $D$-relation;
\item $\bigcup(Y_i: i \in J)= X$;
\item  $(\bigcup(H_i: i \in J), X)$ is a 2-primitive but not 3-transitive Jordan group;
\item $\sigma_j\supseteq \sigma_i |_{Y_j}$ if $i< j$;
\item $\bigcap(\sigma_i: i\in J)$ is equality in $X$;
\item $(\forall g\in G) (\exists i_0 \in J) (\forall i<i_0)(\exists j\in J) (Y_i ^g=Y_j \wedge g^{-1} H_{i} g= H_j)$;
\item for any $x \in X, G_x$ preserves a $C$-relation on $X\setminus \{ x \}$.
\end{enumerate}

We say that $G$ preserves a {\em limit of betweenness relations} if the same condition holds, except that in (ii), `$D$-relation' is replaced by `betweenness relation'. 

\end{mydef}

  We suspect this definition is not optimal, and in particular that the {\em sequence} of  sets $Y_i$ indexed by $J$ should be replaced by a {\em $G$-invariant} collection of sets $Y_i$ which is semilinearly ordered under inclusion -- as happens for the constructions in this paper. As commented in Section 1, the problem then is to show that the classification of primitive Jordan groups from \cite{adeleke1996classification} still holds with this modified definition. 

  Since it is analogous to the above, and will be needed in Section 4, we also give the definition of the notion {\em limit of Steiner systems} which also occurs in the main theorem of \cite{adeleke1996classification}.

  \begin{mydef}(\cite{adeleke1996classification}, Definition 2.1.10) \label{limits-steiner} \em If $(H, X)$ is an infinite Jordan group we say that $H$ preserves a
 {\em limit of Steiner systems} on $X$ if for some $n>2$, $(H,X)$ is $n$-transitive but not $(n+1)$-transitive, and there is a totally ordered index set $(J,\leq)$ with no greatest element, and an increasing chain $(X_j:j\in J)$ of subsets of $X$ such that:
\begin{enumerate}
\item[(i)] $\bigcup(X_j:j\in J)=X$;
\item[(ii)] for each $j\in J$, $H_{\{X_j\}}$ is $(n-1)$-transitive on $X_j$ and preserves a non-trivial Steiner $(n-1)$-system on $X_j$;
\item[(iii)] if $i<j$ then $X_i$ is a subset of a block of the $H_{\{X_j\}}$-invariant Steiner $(n-1)$-system on $X_j$.
\item[(iv)] for all $g\in H$ there is $i_0\in J$, dependent on $h$, such that for every $i>i_0$ there is $j\in J$ such that $X_i^h=X_j$ and the image under $h$ of every block of the Steiner system on  $X_i$ is a  block of the Steiner system on $X_j$;
\item[(v)] for every $j\in J$, the set $X\setminus X_j$ is a Jordan set for $(H,X)$.
\end{enumerate}
\end{mydef}

Next, we give an informal description of the structure $\mathcal{M}_B$ from \cite{bhattmacph2006jordan} -- see e.g. Proposition 5.3. The structure may be viewed as of 
 form $\mathcal{M}_B^0$ having a single ternary relation $L$, though it is in fact built by an amalgamation construction in a rather richer language, so  is called $\mathcal{M}_B^1$ in our notation here. We describe $\mathcal{M}_B$ in terms of relations preserved by the automorphism group. Let $G:=\Aut(\mathcal{M}_B)$.

First, there is (in $\mathcal{M}_B^{\eq}$) a 0-interpretable dense lower semilinear order $(J,<)$ which is a meet-tree; this is a partial order such that for all $a\in J$, $\{x\in J:x<a\}$ is totally ordered by $<$, and such that any two elements of $J$ have a greatest lower bound. For each $a\in J$ the set $C_a$ of `cones' at $a$ is infinite. (Here a {\em cone} is an equivalence class of the equivalence relation $\rho_a$ on $\{x:a\leq x\}$ whereby $\rho_a x y$ if and only if there is $z\in J$ with 
$a<z\wedge z<x\wedge z<y$.) The set $J$ is an index set for a $G$-invariant family $(X_j:j\in J)$ of subsets of $M_B$, with $X_j\supset X_k$ whenever $j<k$, for $j,k\in J$. We have $\bigcup_{j\in J}X_j={M}_B$ and $\bigcap_{j\in J} X_j=\emptyset$. 

The group $G_j$ (the stabiliser of $j$) fixes $X_j$ setwise and preserves a $j$-definable equivalence relation $\sim_j$ on $X_j$. There is a $j$-definable betweenness relation $B_j$ on $Z_j:=X_j/\sim_j$. This is {\em dense} in the sense that for any distinct $x,y\in Z_j$ there is $z\in Z_j\setminus \{x,y\}$ with $B_j(z;x,y)$. It has {\em positive type}, that is, satisfies the axiom (B7) of \cite[Ch. 15]{adeleke1998relations}, namely
$$(\forall x,y,z)(\exists w)(B_j(w;x,y)\wedge B_j(w;x,z) \wedge B_j(w;y,z)).$$
If $a\in Z_j$ there is an equivalence relation $\equiv_a$ on $\{z\in Z_j:z\neq a\}$, whereby, for distinct $x,y\in Z_j\setminus\{a\}$ we have $x\equiv_a y$ if and only if $\neg B_j(a;x,y)$. We shall call the $\equiv_a$-classes {\em branches} at $a$ (they are called {\em sectors} in 
\cite{adeleke1998relations}), and write $W_a$ for the set (which is infinite) of branches at $a$. 

For $a\in X_j$ the set $\{x\in M_B: a\sim_j x\}$  is called a {\em pre-node} of $Z_j$, or {\em pre-node} at $j$, and denoted by $[a]_j$, or $[a]$ when $j$ is clear. The $\sim_j$-class containing $a$ when viewed as an imaginary is called a {\em node} of $Z_j$, and is denoted by $a/\sim_j$. Likewise if $a\in Z_j$ and $U$ is a branch of $Z_j$ at $a$, then $\{x\in X_j: x/\sim_j\in U\}$ is called a {\em pre-branch} of $Z_j$. 

There is for each $i\in J$ an $i$-definable bijection $f_i$ from the set $A_i$ of cones at $i$ to the set $Z_i$. This puts $i$-definably the betweenness structure of $Z_i$ onto the set of cones at $i$. For $i<j_0$ let $C_i(j_0)$ be the cone of $(J,<)$ at $i$ containing $j_0$. If $[a]$ is a pre-node at $j_0$ then there is a unique set $t$ of branches at the node $f_i(C_i(j_0))$ of $Z_i$ such that $[a]$ equals the union of the pre-branches corresponding to members of $t$. We shall write $g_{ij_0}(a)=t$.  In particular, it follows that if $i<j_0$ then
$\bigcup\{g_{ij_0}(a): [a] \mbox{~prenode at~} j_0\}$ is a union of pre-branches at $f(C_i(j_0))$, and is denoted by $g_i(j_0)$.  These maps ensure that, in an iterative way, the betweenness structure of the set $Z_{j_0}$ is imposed on  (a quotient of a subset of) the set of branches at $f_i(C_i(j_0))$. Thus, `higher' sets $Z_{j_0}$ impose structure on `lower' sets $Z_i$.

This gives essentially all the structure.
The relation $L$ is interpreted in $\mathcal{M}_B$ as follows: $\mathcal{M}_B\models L(x;y,z)$ if and only if  there is $j\in J$ such that $x/\sim_j,y/\sim_j,z/\sim_j$ are distinct and the relation $B_j(x/\sim_j;y/\sim_j,z/\sim_j)$ holds. It is shown in  \cite[Section 5]{bhattmacph2006jordan} that all the above structure is interpretable without parameters in  $\mathcal{M}_B^0:=({M}_B,L)$. For example, $J$ is identified with the quotient of a 0-definable subset of $M_B^3$ by a 0-definable equivalence relation. The structure $\mathcal{M}_B^1$ is constructed by Fra\"iss\'e amalgamation of a class of finite structures in a richer finite relational language (though this class is not closed under substructure, and $\mathcal{M}_B^1$ is not homogeneous and so does not have quantifier elimination in this language -- see Example~\ref{nonhomo ex}). It follows from the amalgamation construction that $\mathcal{M}_B^1$ is $\omega$-categorical -- see \cite[Section 4]{bhattmacph2006jordan}.

The basic symmetry properties of $\mathcal{M}_B$ are described below.
\begin{prop} \label{groupM_B} Let $G=\Aut(\mathcal{M}_B^1)$. Then the following hold.
\begin{enumerate}
    \item 
$G$ is 2-primitive and 3-homogeneous but not 3-transitive on ${M}_B$, and is transitive on the set $\{(x,y,z)\in {M}_B^3: \mathcal{M}_B^1\models L(x;y,z)\}$;
\item $G$ is transitive on the set $\{(i,j)\in J^2:i<j\}$;
\item for each $j\in J$, $G_j$ is transitive on the set of triples $\{(x,y,z): B_j([z];[y],[z])\mbox{~holds}\}$, and induces a 2-transitive Jordan group on $Z_j$ in which branches are Jordan sets;
\item each pre-node and pre-branch is a Jordan set for $G$, as is each set $X_j$;
\item the stabiliser in $G$ of any 3-set $\{x,y,z\}$ induces $C_2$ on $\{x,y,z\}$ and fixes $x$ if $L(x;y,z)$ holds (in particular, there is a unique choice of first element such that $L$ holds of the triple). 
\end{enumerate}

\end{prop}

All the symmetry properties apart from those involving Jordan groups follow just from the homogeneity properties ensured by the Fra\"iss\'e construction. The Jordan condition is more subtle, as it involves extending partial isomorphisms with infinite domain, as the complement of the Jordan sets arising  is infinite. One first shows that any pre-node $[a]$ is a Jordan set, essentially by showing that a certain group  embeds in $G$ which fixes the complement of $[a]$ in $M_B$ and acts on $[a]$ as an iterated wreath product in the manner, for example, of \cite[Section 6]{cameron1987some}. The point here is essentially that if $[a]$ is a pre-node  at $j_0$ and $i<j_0$, then because $[a]$ is the union of a set $t$ of pre-branches at $f(C_{i}(j_0))$, there are two $ij_0 a/\sim_{j_0}$-definable equivalence relations $E_i$ and $F_i$ on $[a]$. The $F_i$-classes are the pre-branches at $f(C_j(i))$ contained in $t$, and the $E_i$-classes are the pre-nodes at $i$ contained in $t$. In particular, $E_i$ refines $F_i$ and for $i<k<j_0$ we have that $F_k\supset E_k\supset F_i\supset E_i$. There is an 
$a/E_{j_0}$-definable $C$-relation $C_{[a]}$ on $[a]$, with
$C_{[a]}(x;y,z)$  holding if either $y=z\neq x$ or $x,y,z$ are distinct and for some $i<j_0$, some $F_i$-class contains $y,z$ and omits $x$.

\begin{figure}
	\begin{center}
			\includegraphics[scale=.7]{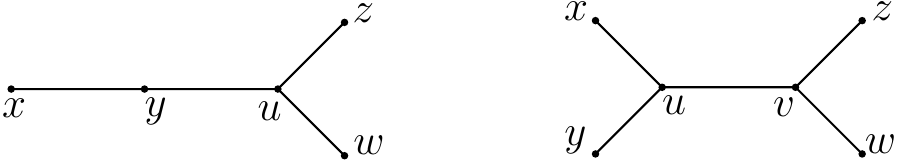}
		\end{center} 
  \caption{The configurations witnessing $N(x,y;z,w)$ on the left and $S(x,y;z,w)$ on the right.}
  \label{fig:NS}
\end{figure}

In \cite{bhattmacph2006jordan}, the actual amalgamation construction is carried out in a richer language with, as well as the ternary relation symbol $L$, three arity 4 relation symbol $L', N, S$. In $\mathcal{M}_B^1$, $L'(x;y,z;w)$ holds if there is $j\in J$ such that $w\not\in X_j$ and $x,y,z$ are $\sim_j$-inequivalent elements of $X_j$ with the relation
$B_j(x/\sim_j,y/\sim_j,z/\sim_j)$ holding. We have $\mathcal{M}_B^1\models N(x,y;z,w)$ if there is $u\in {M}_B$ and some $j\in J$ with $x,y,z,w,u$ all $\sim_j$-inequivalent such that their $\sim_j$-classes carry the  structure in $(Z_j,B_j)$ depicted in Figure \ref{fig:NS}. Likewise, $\mathcal{M}_B^1\models S(x,y;z,w)$ if there are $u,v\in \mathcal{M}_B$ and some $j\in J$ with $x,y,z,w,u,v$ all $\sim_j$-inequivalent and as  depicted in Figure \ref{fig:NS} in $(Z_j,B_j)$. 

For the structure $\mathcal{M}_D$, essentially the same assertions hold, except that the ternary betweenness relation $B_j$ on $Z_j$ is replaced  by an (arity four) $D$-relation $D_j$, and pre-nodes are replaced by `pre-directions'. There are some differences in the proof, and also some inaccuracies in the proof in \cite{bhattmacph2006jordan} are corrected in \cite{almazaydeh2021jordan}. 

A $D$-set $(N,D)$ can be viewed as a set of `directions' of an underlying (and interpretable) $B$-set $(P,B)$, where $P$ is a 0-interpretable quotient of $N^3$. We shall refer to the elements of a $B$-set as `vertices', and talk of `branches' at a vertex as above, except that now a branch at a  vertex $a$ of such $P$ may be viewed either as a subset of $P$, or (in a natural way) as a subset of $N$. Given distinct $x,y,z\in N$, we write $\ram(x,y,z)$ for the unique node $a$ of $P$ such that $x,y,z$ lie in different branches at $a$. The structure $\mathcal{M}_D$ has an interpretable semilinear order $(J,\leq)$
indexing a family $(X_j:j\in J)$ of subsets of $\mathcal{M}_D$ semilinearly-ordered by reverse inclusion, again with $\bigcup_{j\in J} X_j=\mathcal{M}_D$ and $\bigcap_{j\in J} X_j=\emptyset$.  Again, there is a $j$-definable equivalence relation $\sim_j$ on $X_j$, and $Z_j=X_j/\sim_j$ carries the structure of a $D$-relation $D_j$. However, there is an extra phenomenon: if $(P_j,B_j)$ is the betweenness relation arising as above from the $D$-set $(Z_j,D_j)$, and $a\in P_j$, then there is a unique  $(a,j)$-definable branch $U_j$ at $a$, called the `special branch' at $a$.

In \cite{almazaydeh2021jordan} the amalgamation construction of $\mathcal{M}_D$ takes place in a language $\mathcal{L}_D^1$ which has a ternary relation symbol $L$, relation symbols $L',S$ of arity 4, and  $S'$ of arity 5, $R$ of arity 6, and $Q$ of arity 7. In the structure $\mathcal{M}_D^1$ in this language, we write $L(x;y,z)$ if there is $j\in J$ with $x,y,z\in X_j$ inequivalent modulo $\sim_j$ such that $x/\sim_j$ lies in the special branch at $\ram(x/\sim_j,y/\sim_j,z/\sim_j)$; we put $L'(x;y,z;w)$ if in addition such $X_j$ does not contain $w$. We put $\mathcal{M}_D^1\models S(x,y;z,w)$ if there is $j\in J$ with $x,y,z,w$ distinct in $X_j$ modulo $\sim_j$, and with $D_j(x/\sim_j,y/\sim_j;z/\sim_j,w/\sim_j)$ holding; we have $\mathcal{M}_D^1\models S'(x,y;z,w;u)$ if the above holds and in addition $u\not\in X_j$. In the above setting, we say that $L(x;y,z)$ and $S(x,y;z,w)$ are {\em witnessed} in the $D$-set $(Z_j,D_j)$. Now $\mathcal{M}_D^1\models R(x;y,z:u;v,w)$ if $L(x;y,z)$ and $L(u;v,w)$ hold and are witnessed in the same $D$-set, and likewise $Q(x,y;z,w:p;q,s)$ holds if $S(x,y;z,w)$ and $L(p;q,s)$ hold and are witnessed in the same $D$-set. 

In \cite[Lemma 4.1]{almazaydeh2021jordan}, it is shown that $L',S',Q,R$ are 0-definable in $({M}_D,L,S)$. In fact, it is easily seen that $S$ is 0-definable in $({M}_D,L)$, though we omit the details (it will in fact follow from Theorem~\ref{max}). Thus, we may view $\mathcal{M}_D$ as a structure, denoted $\mathcal{M}_D^0$, with just the single ternary relation symbol $L$. 

The analogue of Proposition~\ref{groupM_B} for $\mathcal{M}_D$ is the following. 
\begin{prop} \label{groupM_D} Let $G=\Aut(\mathcal{M}_D^1)$. Then the following hold.
\begin{enumerate}
    \item 
$G$ is 2-primitive and 3-homogeneous but not 3-transitive on ${M}_D$, and is transitive on the set $\{(x,y,z)\in M_D^3: \mathcal{M}_D^1\models L(x;y,z)\}$;
\item $G$ is transitive on the set $\{(i,j)\in J^2:i<j\}$;
\item for each $j\in J$, $G_j$ is transitive on the set of triples $\{(x,y,z): L(x;y,z)\mbox{~holds witnessed in~} Z_j\}$, and induces a 2-transitive Jordan group on $Z_j$ in which branches are Jordan sets;
\item each pre-direction and pre-branch is a Jordan set for $G$, as is each set $X_j$;
\item the stabiliser in $G$ of any 3-set $\{x,y,z\}$ induces $C_2$ on $\{x,y,z\}$ and fixes $x$ if $L(x;y,z)$ holds (in particular, there is a unique choice of first element such that $L$ holds of the triple).
\end{enumerate}

\end{prop}

We comment briefly on the amalgamation constructions of $\mathcal{M}_B^1$ and $\mathcal{M}_D^1$. In the case of $\mathcal{M}_B^1$, a certain class $\mathcal{B}^1$ of finite $\mathcal{L}_B^1$-structures, namely 
 `trees of $B$-sets' is identified. We refer to \cite{bhattmacph2006jordan} for details, but roughly, such a structure is isomorphic to a finite substructure $K$ of $\mathcal{M}_B^1$ such that the finite $B$-sets have `positive type', that is, if $x,y,z$ are distinct elements of such a $B$-set then there is $t\in K$ in the $B$-set lying between any two of $x,y,z$. It is shown that the class $\mathcal{B}$ has the amalgamation property, and $\mathcal{M}_B^1$ is the Fra\"iss\'e limit. The construction of $\mathcal{M}_D^1$ is similar, obtained by amalgamating a class $\mathcal{D}$ of finite $\mathcal{L}^1_D$-structures. The members of $\mathcal{D}$ are isomorphic to finite substructures $K$ of $\mathcal{M}_D^1$ such that any node of any (finite) $D$-set has a special branch within the structure $K$. See \cite[Section 3]{almazaydeh2021jordan} for a detailed description of $\mathcal{D}$. It follows from the construction of $\mathcal{M}_B^1$ that any isomorphism between substructures of $\mathcal{M}_B^1$ which lie in $\mathcal{B}$ extends to an automorphism of $\mathcal{M}_B^1$. The corresponding assertion holds for $\mathcal{M}_D^1$.

\section{Proofs of Proposition~\ref{homogenizable} and Theorem~\ref{growth}}

We first prove the non-homogeneity assertion in Proposition~\ref{homogenizable}(i), through the following two examples. The essential point is that though the structures are Fra\"iss\'e limits of  classes of finite structures with the amalgamation property, these classes are not closed under substructure.

\begin{ex}\label{nonhomo ex}
\em In the structure $\mathcal{M}_D^1$, consider the set
 $C=\{x,y,z,w,u\}$ and $C'=\{x',y',z',w' ,u'\}$ and consider the two structures $A$ and $A'$ depicted in Figures \ref{fig:A} and \ref{fig:A'} respectively, with $C<A$ and $C'<A'$. We may view $A$ and $A'$ as substructures of $\mathcal{M}_D^1$. It can be checked that the only $\mathcal{L}_D^1$-relations holding on $C$ and $C'$ are $L$ and $S$, and that the map $h: (x,y,z,w,u)\to (x',y',z',w',u')$ is an isomorphism. However, $h$ does not extend to an automorphism of $\mathcal{M}_D^1$, since the relations $S(z,y;u,x)$ and $S(z,w;x,y)$ are witnessed in different   $D$-sets in $\mathcal{M}_D^1$ whereas $S(z',y';u',x')$ and $(z',w';x',y')$ are witnessed in the same $D$-set. 
 In the diagrams, which adopt the same conventions as \cite[Figure 2]{almazaydeh2021jordan}, the arrow indicates which branch is special at each node. 

  A similar construction shows that $\mathcal{M}_B^1$ is not homogeneous, but we omit the details. Figures \ref{fig:B1} and \ref{fig:B2} illustrate two configurations for $(a,b,c,d,e,f)$ in $\mathcal{B}^1$ which are isomorphic but not in the same orbit.

  \begin{figure}
	\begin{center}
			\includegraphics[scale=.6]{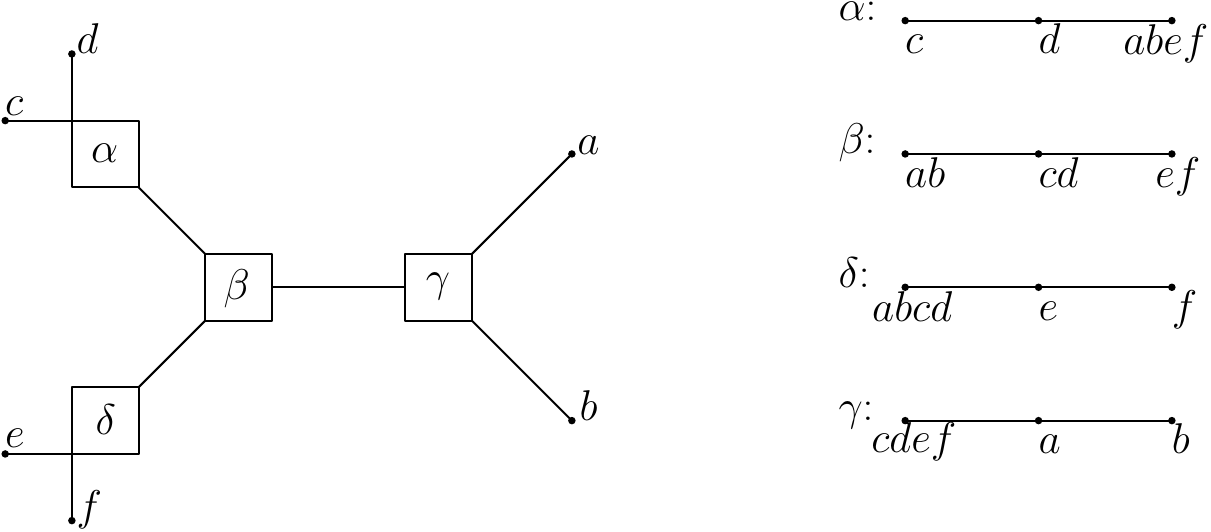}
		\end{center} 
  \caption{The first of two non-isomorphic substructures of $\mathcal{M}_B^1$ in the same orbit.}
  \label{fig:B1}
\end{figure}

  \begin{figure}
	\begin{center}
			\includegraphics[scale=.6]{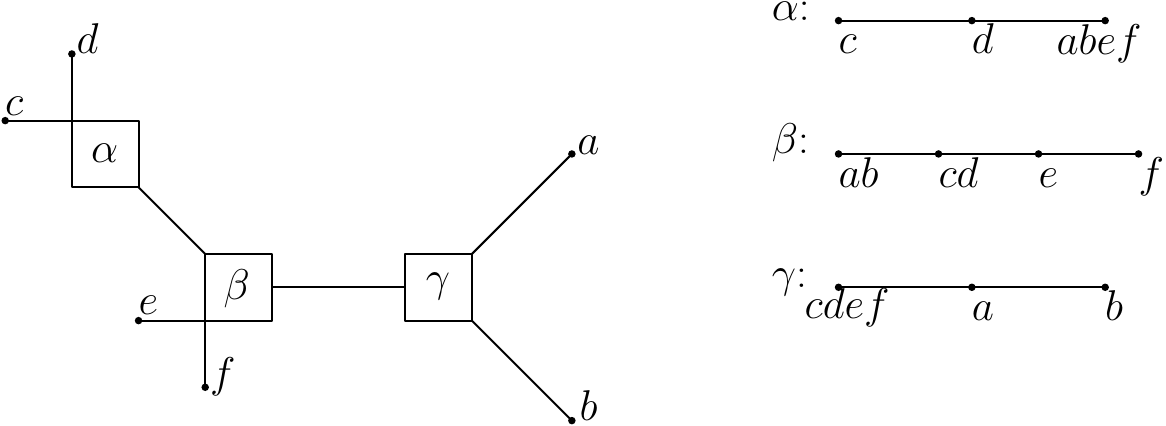}
		\end{center} 
  \caption{The second of two non-isomorphic substructures of $\mathcal{M}_B^1$ in the same orbit.}
  \label{fig:B2}
\end{figure}

\end{ex}

  \begin{figure}
	\begin{center}
			\includegraphics[scale=.9]{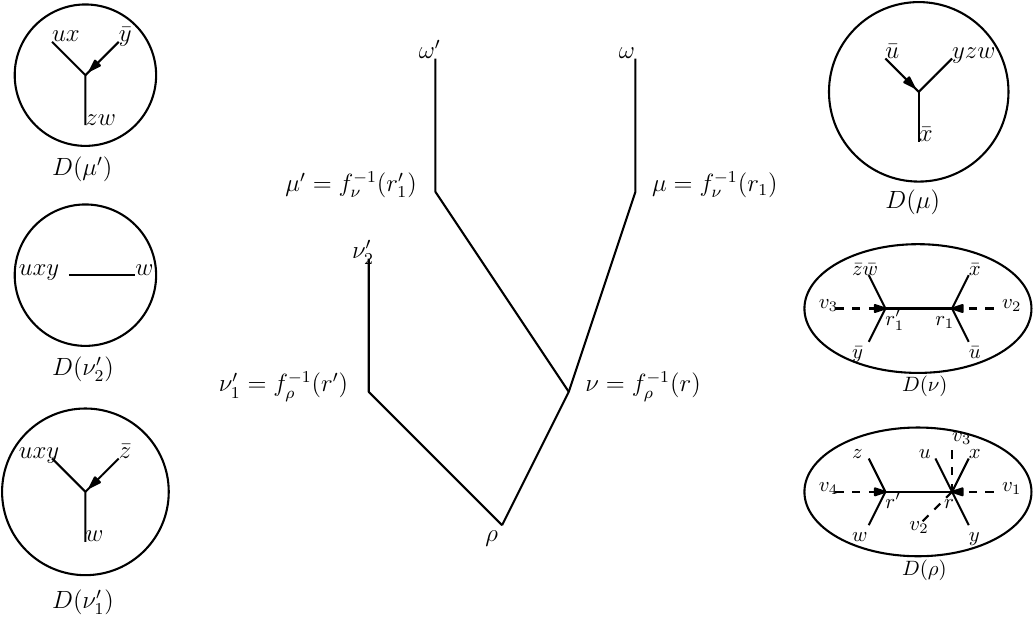}
		\end{center} 
  \caption{The structure $A \subset \mathcal{M}_D^1$.}
  \label{fig:A}
\end{figure}

  \begin{figure}
	\begin{center}
			\includegraphics[scale=.9]{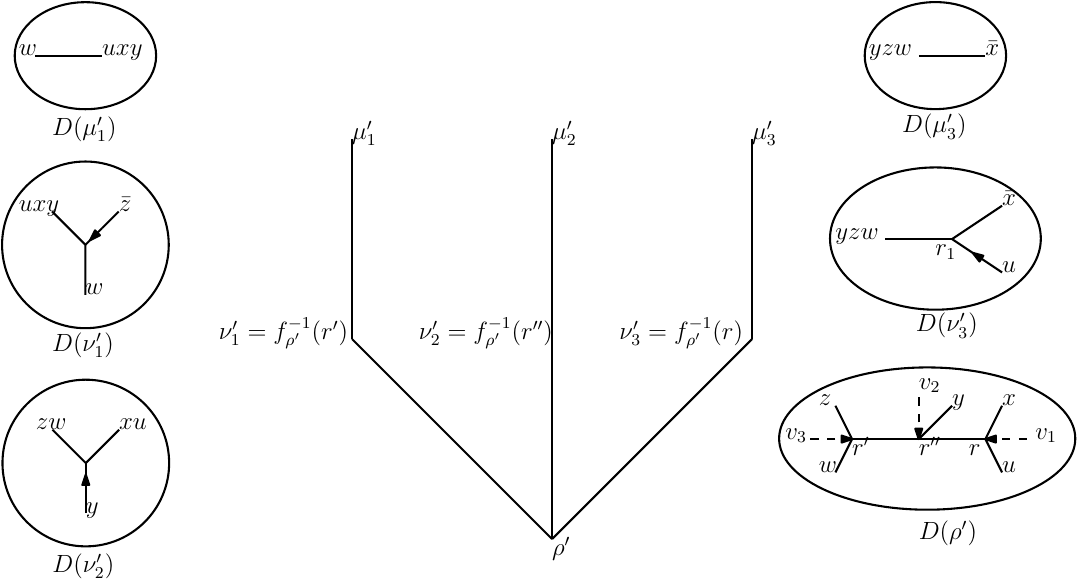}
		\end{center} 
  \caption{The structure $A' \subset \mathcal{M}_D^1$.}
  \label{fig:A'}
\end{figure}

Next, we prove homogenizability of $\mathcal{M}_D^1$. We define a relational language $\mathcal{L}_D^2$ with the relation symbols $L$ and  $S$  interpreted on  ${M}_D$ as before, but, in place of the symbols
$L',S', Q, R$ of $L_D$, the language $\mathcal{L}_D^2$ has symbols $P$ of arity 6, $Q^{\leq}$ and $Q^{\geq}$ of arity 7,  and $T$ of arity 8, interpreted on ${M}_D$ as follows: we have $P(x;y,z:p;q,s)$ if $L(x;y,z)$ holds witnessed at $i$, $L(p;q,s)$ holds witnessed at $j$, and $i\leq j$. Likewise, 
$T(x,y;z,w: p,q;s,t)$ holds if $S(x,y;z,w)$ holds witnessed at $i$, $S(p,q;s,t)$ holds witnessed at $j$, and $i\leq j$. Similarly
$Q^{\leq}(x,y;z,w:p;q,s)$ if $S(x,y,z,w)$ is witnessed at $i$ and $L(p;q,s)$ is witnessed at $j$ with $i\leq j$, and $Q^{\geq}(x,y;z,w:p;q,s)$ if $S(x,y,z,w)$ is witnessed at $i$ and $L(p;q,s)$ is witnessed at $j$ with $i\geq j$.

\begin{lem} \label{MDhom}
    \begin{enumerate}
        \item $\Aut(\mathcal{M}_D^1)=\Aut(\mathcal{M}_D^2)$.
        \item The $\mathcal{L}_D^2$-structure $\mathcal{M}_D^2$ is homogeneous.
        \end{enumerate}
\end{lem}
\begin{proof}
    \begin{enumerate}
        \item It is clear from our description above of $\mathcal{M}_D^1$ that $G=\Aut(\mathcal{M}_D^1)$ preserves $Q^{\leq}$, $Q^{\geq}$, $P$ and $T$. For the converse, observe that for distinct $x,y,z,t\in M_D$, $L'(x;y,z;t)$ holds if and only if $L(t;y,z)$ is witnessed in a $D$-set strictly below  that witnessing $L(x;y,z)$, and this information is given by $P$. Likewise, $S'(x,y;z,w;t)$ holds if and only if $Q^{\geq}$ holds with $(x,y,z,w)$ in the first four entries and some ordering of $\{z,w,t\}$ in the last three. Finally, we have 
        $R(x;y,z;p;q,s)$ if and only if $P(x;y,z:p;q,s)$ and $P(p;q,s:x;y,z)$ both hold, and $Q(x,y;z,w:p;q,s)$ holds if and only if $Q^{\geq}$ and $Q^{\leq}$ both hold of $(x,y,z,w,p,q,s)$. 
        \item It suffices by the partial homogeneity of $\mathcal{M}_D^1$ to prove the following assertion. Let
        $A_1,A_2$ be finite subsets of $M_D$ and $f:A_1\to A_2$ an isomorphism of the induced substructures of $M_D^2$. Then there are finite $A_1'\supseteq A_1$ and $A_2'\supseteq A_2$ such that $f$ extends to an isomorphism between the substructures of $\mathcal{M}_D^1$ induced on $A_1'$ and $A_2'$ respectively, and these $\mathcal{L}_D^1$-structures on $A_1'$ and $A_2'$ lie in $\mathcal{D}$. 

        We build $A_1'$ and $A_2'$ as follows. Using $P$, $Q^{\leq}$, $Q^{\geq}$, and $T$,  we may identify which tuples satisfying $S$ or $L$ are witnessed in $A_1$ and $A_2$ respectively at the lowest level (indexed in $J$ by $i_1$ and $i_2$ respectively). These determine a `lowest' $D$-set of $A_1$ and $A_2$ respectively, and $f$ induces an isomorphism of these $D$-sets, and hence of the corresponding $B$-sets interpretable in them. Using $Q^{\leq}, Q^{\geq}$ we may  also identify which $L$-relations are witnessed in these lowest $D$-sets, and hence identify the special branches at certain nodes. The isomorphism $f$ induces a bijection between the nodes of $A$ with no special branches (in this lowest $D$-set) and the corresponding nodes of $A_2$. We add to $A_1$ and $A_2$ a direction (in the special branch) at each of these nodes. The map $f$ extends to these expanded structures. We now continue inductively. For each point $a$ of the $B$-set induced by $A_1$ at level $i_1$, $f$ induces a map from the cone at $i_1$ corresponding to $a$ to the cone at $f(i_1)$ corresponding to $f(a)$. We may identify those tuples of $A_1$ which satisfy $S$ or $L$ witnessed at the lowest level within this cone at $i_1$, thereby identifying a $D$-set of $A_1$ directly above the lowest level, and $f$ maps this isomorphically to the corresponding $D$-set of $A_2$. Again, we may add  elements to $A_1$ and $A_2$ to ensure that each node of the corresponding $B$-sets has a special branch. Repeating this argument, we eventually obtain an isomorphism $f:A_1'\to A_2'$, as claimed.

    \end{enumerate}
\end{proof}
We give without details a similar result for $\mathcal{M}_B$, namely the following.
\begin{lem} \label{MBhom}
The structure $\mathcal{M}_B^1$ is homogenizable.
\end{lem}
\begin{proof} We just give a sketch.  We replace the original language $\mathcal{L}_B^1$ by a language $\mathcal{L}_B^2$ which has the relation symbols, $L, N,S$ but not $L'$, and also has relation symbols which express that if the $B$-set indexed by $i$ witnesses $L$, $N$ or $S$ for some tuple, and the $B$-set indexed by $j$ witnesses $L,N$ or $S$ for some other tuple, then $i\leq j$. It can be checked that these relations are all 0-definable in $\mathcal{M}_B^1$ and that 
$L'$ is 0-definable from them. Furthermore, given a finite subset $A$ of $M$, we can identify the lowest $B$-set of $\mathcal{M}_B^1$ in which a tuple from $A$ satisfies one of $L,N$ or $S$, and which tuples satisfy such relations in this $B$-set. Using these relations, we may uniquely reconstruct this lowest $B$-set restricted to $A$. It may not have positive type, i.e. may miss certain ramification points, but we can add these in a canonical manner. We  proceed inductively to embed $A$ canonically in a substructure of $\mathcal{M}_B^1$ lying in $\mathcal{B}$.
Now if $f:A_1\to A_2$ is an isomorphism of finite substructures of $\mathcal{M}_B^2$, then $f$ extends to an isomorphism $f':A_1' \to A_2'$ of finite substructures of $\mathcal{M}_B^1$ lying in $\mathcal{B}$, and the homogeneity properties given by the amalgamation construction of $\mathcal{M}_B^1$ ensure that $f'$ extends to an automorphism of $\mathcal{M}_B^1$. 
\end{proof}

\begin{figure}
	\begin{center}
			\includegraphics[scale=.7]{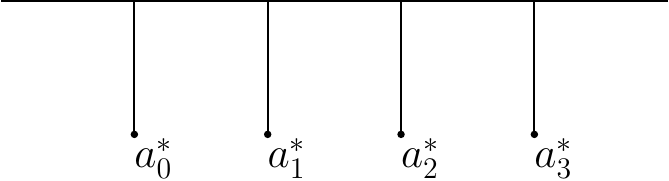}
		\end{center} 
  \caption{A $D$-set indiscernible sequence.}
  \label{fig:Dind}
\end{figure}

Before proving that $\mathcal{M}_B$ and $\mathcal{M}_D$ are NIP, we shall obtain a description of indiscernible sequences of singletons, useful also for later results.

\begin{lem} \label{ind-D}
Let  $I=(a_i:i\in \omega)$ be an indiscernible  sequence of singletons in $\mathcal{M}_D$. Then one of the following holds.
\begin{enumerate}
    \item[(i)] $I$ is a {\em $D$-set indiscernible}; that is, there is $j\in J$ indexing the $D$-set $(Z_j,D_j)$ such that if $a_i^*:=a_i/\sim_j$ for each $i\in \omega$, then the $a_i^*$ are in $Z_j$ as depicted in Figure \ref{fig:Dind}. That is, for any $i<j<k<l$ we have $S(a_i,a_j;a_k,a_l)$ witnessed in $Z_j$. 
    \item[(ii)] $I$ is an {\em $S$-free indiscernible}, meaning that no quadruple from $I$ satisfies $S$. Now one of:
    
    (a) $I$ is {\em upwards $S$-free}; that is, 
there is an increasing sequence
       $(n_i:i\in \omega)$ in $J$ such that  for any  $i_1,i_2,i_3\in \omega$ with $i_1<i_2<i_3$, the relation $L(a_{i_1};a_{i_2},a_{i_3})$ holds and is witnessed at $n_{i_1}$. 

(b) $I$ is {\em downwards $S$-free}; that is, there is a decreasing sequence
       $(n_i:i\in \omega, i\geq 3)$ in $J$ such that  for any  $i_1,i_2,i_3\in \omega$ with $i_1<i_2<i_3$, the relation $L(a_{i_3};a_{i_1},a_{i_2})$ holds and is witnessed at $n_{i_3-2}$.

    \item[(iii)]  $I$ is a {\em mixed indiscernible}; that is, there is a decreasing sequence $n_1>n_2>\ldots$ in $J$ such that for each $i\geq 1$ we have: $a_0,\ldots,a_{i+1}$ all lie in distinct branches at the vertex $v_i$ of $Z_{n_i}$, $a_j\sim_{n_i} a_{i+2}$ for $j>i+2$, and $S(a_p,a_q;a_{i+1},a_r)$ is witnessed in $Z_{n_i}$ for all $r>i+1$ and distinct $p,q\leq i$.

\end{enumerate}
\end{lem} 

\begin{proof} The argument is elementary but there are many cases to consider, so we only sketch the idea. 

Exactly one of $L(a_0;a_1,a_2)$, $L(a_2;a_0,a_1)$ or $L(a_1;a_0,a_2)$ holds. Thus, by indiscernibility exactly one of the following three cases holds. 

{\em Case 1.} Whenever $i,j,k\in \omega$ are distinct with $i<\Min\{j,k\}$ we have $L(a_i;a_j,a_k)$. We suppose that $L(a_0;a_1,a_2)$ is witnessed in $Z_{n_0}$, with $\ram(a_0,a_1,a_2)=v_0$.

{\em Case 2.} Whenever $i,j,k\in \omega$ are distinct with $k>\Max\{i,j\}$ we have $L(a_k;a_i,a_j)$.

{\em Case 3.} Whenever $i<j<k$ or $k<j<i$ we have $L(a_j;a_i,a_k)$.

{\em Case 1a.} Suppose $L(a_0,a_1,a_2)$ is witnessed in $Z_{n_0}$ and $a_3$ is $\sim_{n_0}$-inequivalent to $a_0,a_1,a_2$. Then 
automorphisms inducing $(a_0,a_1,a_2,a_3)\mapsto (a_0,a_1,a_2,a_i)$ for $i>3$ show that $L(a_0;a_1,a_i)$ is witnessed in $Z_{j_0}$, and maps $(a_0,a_1,a_2,a_3)\mapsto (a_0,a_1,a_i,a_j)$ (for $3\leq i<j$) show that all the $a_i$ are $\sim_{n_0}$ inequivalent and meet at the same vertex $v_0$ of $Z_{n_0}$. Using indiscernibility, there is $n_1>n_0$ such that the same configuration holds in $Z_{n_1}$ (but with $a_0$ omitted, and with $a_1$ in the special branch at the central vertex $v_1$). Now iterate this argument, to obtain that $I$ is an upwards $S$-free indiscernible. 

{\em Case 1b.} $a_3\not\in X_{n_0}$. In this case, $L(a_3;a_0,a_1)$ is witnessed in a $D$-set below $Z_{n_0}$, contrary to indiscernibility. 

{\em Case 1c.} $a_3$ lies in the same branch at $v_0$ in $Z_{n_0}$ as $a_0$. In this case $L(a_3;a_1,a_2)$ holds witnessed in $Z_{n_0}$, contradicting $L(a_1;a_2,a_3)$.

{\em Case 1d.} $a_1\not\sim_{n_0} a_3$, and $a_3$ lies in the same branch at $v_0$ in $Z_{n_0}$ as $a_1$. In this case the $D$-set which witnesses $L(a_1;a_2,a_3)$ must also witness $L(a_1;a_0,a_3)$ contradicting that $L(a_0;a_1,a_3)$ holds by indiscernibility.

{\em Case 1e.} $a_1\sim_{n_0} a_3$.
Now a $D$-set below $Z_{n_0}$ will have the $D$-set configuration of Case 1d and is eliminated for the same reason.

{\em Case 1f.} $a_3\not\sim_{n_0} a_2$ and $a_3$ lies in the same $Z_{n_0}$ branch as $v_0$ as $a_2$. Now an automorphism extending $(a_0,a_1,a_2,a_3)\mapsto (a_0,a_1,a_3,a_4)$ ensures that for distinct $p,q\in \{0,1,2\}$, $S(a_p,a_q;a_3,a_4)$ holds, witnessed in $Z_{n_0}$, and iteration of the argument shows that $I$ is a $D$-set indiscernible.

{\em Case 1g.} $a_3\sim_{n_0} a_2$. Now for $i>3$ an automorphism inducing
$(a_0,a_1,a_2,a_3)\mapsto (a_0,a_1,a_2,a_i)$  fixes the $D$-set witnessing $L(a_0;a_1,a_2)$, so ensures that $a_2\sim_{n_0} a_i$ holds. Now there is $n_1<n_0$ in $J$ and \ vertex $v_1$ in $Z_{n_1}$ such that $a_0,a_1,a_2$ are in distinct branches at $v_1$, and $a_3$ in the same branch as $a_2$ at $v_1$, with $a_3\not\sim a_2$. Consideration of where $L(a_1;a_2,a_3)$ is witnessed yields that $a_i\sim_{n_1} a_3$ for $i>3$, and it is easy to recover that $I$ is  mixed indiscernible. 

{\em Case 2a.} $a_0,a_1,a_2,a_3$ lie in distinct branches at $v_0$. In this case $L(a_2;a_3,a_0)$ holds, contradicting that $L(a_3;a_0,a_2)$ holds by indiscernibility. 

{\em Case 2b.} $a_3\not\in X_{n_0}$. Now there is $n_1>n_0$ in $J$ so that $a_0,a_1,a_2,a_3$ lie in distinct branches at vertex $v_1$ in $Z_{n_1}$  with $a_3$ in the special branch, and we obtain that $I$ is a downwards $S$-free indiscernible.

{\em Case 2c.} $a_3$ lies in the same branch at $v_0$ as $a_1$ or $a_2$. This would imply $L(a_2;a_3,a_1)$ or $L(a_2;a_3,a_0)$, both of which are false by indiscernibility.

{\em Case 2d.} $a_3\not\sim_{n_0} a_2$ but $a_2,a_3$ lie in the same branch at $v_0$. Arguments similar to Case 1f yield that $I$ is a $D$-set indiscernible. 

{\em Case 2e.} $a_2\sim_{n_0} a_3$. Now $I$ is a mixed indiscernible, as in Case  1g.

{\em Case 3a.}  $a_0,a_1,a_2,a_3$ lie in distinct branches at $v_0$. This yields that $L(a_1;a_2,a_3)$ holds, contradicting that $L(a_2;a_1,a_3)$ holds by indiscernibility.

{\em Case 3b.} $a_3\not\in X_{n_0}$. In this case, $L(a_3;a_0,a_1)$ is witnessed in a $D$-set below $X_{n_0}$, contradicting $L(a_1;a_0,a_3)$.

{\em Case 3c.} $a_3$ lies in the same branch at $v_0$ as $a_1$. Then $L(a_3;a_0,a_2)$, contradicting $L(a_2;a_0,a_3)$. 

{\em Case 3d.} $a_3$ lies in the same branch at $v_0$ as $a_0$. Then $L(a_1;a_2,a_3)$ holds, contradicting $L(a_2;a_1,a_3)$.

{\em Case 3e.} $a_3\sim_{n_0} a_2$. Now $I$ is a mixed indiscernible, much as in Case 1g.

{\em Case 3f.} $a_3\not\sim_{n_0}$ and $a_3,a_2$ lie in the same branch at $v_0$. Arguments similar to Case 1f show that $I$ is a $D$-set indiscernible.
\end{proof}

The following lemma gives a similar description of indiscernible sequences of singletons in $\mathcal{M}_B$. 

\begin{lem} \label{ind-B}

Let  $I=(a_i:i\in \omega)$ be an indiscernible sequence of singletons in $\mathcal{M}_B$. Then one of the following holds.
\begin{enumerate}
    \item[(i)] $I$ is a {\em $B$-set indiscernible}; that is, there is a $B$-set $(Z_j,B_j)$ such that one of the following holds.

    (a) $Z_j$ witnesses that $L(a_q;a_p,a_r)$ whenever $p,q,r\in \omega$ with $p<q<r$.

    (b) There are $(b_i:i\in \omega)$ such that $Z_j$ witnesses that $L(b_q;b_p,b_r)$ whenever $p,q,r\in \omega$ with $p<q<r$, and furthermore, $Z_j$ witnesses that $L(b_p;a_p,b_q)$ for any $p\neq q$ in $\omega$.

    \item[(ii)] One of the following holds. 
    
    (a) $I$ is an {\em upwards star-sequence}; that is, 
there is an increasing sequence
       $(n_i:i\in \omega)$ in $J$ such that in $Z_{n_i}$ the $a_j$ for $j>i$ lie in distinct branches at $a_i$, and $a_k\not\in X_{n_i}$ for $k<i$.

       (b) $I$ is a {\em downwards star-sequence}; this means that there is an decreasing sequence
       $(n_i:i\in \omega)$ in $J$ such that in $Z_{n_i}$ the elements $a_0,\ldots a_{i+1}$ lie in distinct branches at $a_{i+2}$, and $a_j\not\in X_{n_i}$ for $k>i+2$.

    \item[(iii)] One of the following holds.
    
    (a) $L(a_2;a_0,a_1)$ and $I$ is a {\em downwards mixed indiscernible}; that is, there is a decreasing sequence
    $n_1>\ldots$ in $J$ such that in $Z_{n_i}$ the elements $a_0,\ldots,a_{i+1}$ are $\sim_{n_i}$-inequivalent, the $a_j$ are pairwise $\sim_{n_i}$ equivalent for $j\geq i+1$, and there is a vertex  $v_i$ of $Z_{n_i}$ such that $a_0,\ldots  a_i$ lie in different branches at $v_i$ and $a_i,a_{i+1}$ lie in the same branch at $v_i$.

    (b) $I$ is an {\em upwards mixed indiscernible}; that is, there is an increasing sequence $n_0<n_1<\ldots$ in $J$, and $a_0,\ldots,a_{i-1}$ are $\sim_{n_i}$-equivalent and  the $a_j$ are $\sim_{n_i}$ inequivalent for $j\geq i-1$, and one of: 
    (I)
    $L(a_0;a_1,a_2)$,  there is one branch at $a_0$ containing $a_i$, and another branch at $a_0$ containing a vertex $v_i$ and $a_p$ for all $p>i$, with $a_0, a_{i+1},a_{i+2},\ldots $ all in different branches at $v_i$; or (II) $L(a_1;a_0,a_2)$, there is one branch at $a_i$ containing  $a_0$, and another branch at 
    $a_i$ containing a vertex $v_i$ and $a_p$ for all $p>i$, with $a_0, a_{i+1},a_{i+2},\ldots $ all in different branches at $v_i$;
    in $Z_{n_i}$.

\end{enumerate}
\end{lem}

\begin{proof} This is very similar to the proof of Lemma~\ref{ind-D}, and we omit the details.
\end{proof}

\begin{lem} \label{NIPtheory}
The structures $\mathcal{M}_B$ and $\mathcal{M}_D$ have NIP theory.    
\end{lem}
\begin{proof}
   Again, we focus on $\mathcal{M}_D$, just mentioning how to adjust the argument for $\mathcal{M}_B$. Since we use quantifier-elimination, we work with $\mathcal{M}_D^2$, i.e. in the language with quantifier-elimination. Let $(a_i:i\in \omega)$ be a sequence of indiscernibles of singletons. Given that any finite Boolean combination of NIP formulas is NIP, it suffices to show that the atomic formulas  of $\mathcal{L}_D^2$ are NIP. For example, we must show that a formula of the form
$L(x;b_1,b_2)$ cannot define the set
$\{a_{2i}:i\in \omega\}$, with a similar statement for the formula $L(b_1;x,b_2)$, and for the other atomic formulas of $\mathcal{L}_D^2$. All such statements are clear by inspection of the three types of indiscernible sequences. We omit the details.

In the case of $\mathcal{M}_B^2$, the statement  of the Claim is unchanged, and the rest of the argument is essentially as for $\mathcal{M}_D^2$.
\end{proof}

\begin{lem}\label{trivialacl}
    The structures $\mathcal{M}_D$ and $\mathcal{M}_B$ have trivial algebraic closure, that is, for any subset $A$ of the domain we have $\acl(A)=A$.
\end{lem}

\begin{proof}
We work  in $\mathcal{M}_D$ but the argument is the same in  $\mathcal{M}_B$. We may suppose that $A$ is finite. Suppose $a\in M_D\setminus A$ and let $A'=A\cup\{a\}$. There is a largest $j\in J$ such that the elements $x/\sim_j$ for $x\in A'$ are all distinct. For any $k$, we  choose $a_1=a,\ldots,a_k \in a/\sim_j$ all realising $\tp(a/A)$. It follows that $a\not\in \acl(A)$.
    
\end{proof}

{\em Proof of Proposition~\ref{homogenizable}.} This follows immediately from Example~\ref{nonhomo ex} and Lemmas~\ref{MDhom}, \ref{MBhom}, \ref{NIPtheory} and \ref{trivialacl}. \qed

\medskip

We turn next to the proof of Theorem~\ref{growth}. We consider first a structure $M_{C,\lev}$ which arises under the name $\partial T(\prec)$ in \cite[Section 6]{cameron1987some}. Let $X$ be a countably infinite set with a distinguished element $e$, and $L_{C,\lev}$ be a language with a ternary relation $C$ and arity 4 relation $V(x,y;z,w)$. The domain of $M_{C,\lev}$
is the set of sequences $a=(a(i))_{i\in \mathbb{Q}}$ where  the $a(i)$ lie in $X$, and $a$ has finite support in the sense that all but finitely many $a(i)$ are equal to $e$. We put $C(a;b,c)$ if either $b=c$ and $a\neq b$, or $a,b,c$ are distinct and $\Min\{i: a(i)\neq b(i)\}<\Min\{i:b(i)\neq c(i)\}$.  Then $C$ is a $C$-relation  on $M_{C,\lev}$, and is in fact the universal homogeneous $C$-relation, as described in Section 4 of \cite{cameron1987some} where it is denoted by $\partial T$. If $a,b,c,d\in M_{C,\lev}$, we put 
$M_{C,\lev}\models V(a,b;c,d)$ if $a\neq b$, $c\neq d$, and $\Min\{i:a(i)\neq b(i)\}\leq \Min\{i: c(i)\neq d(i)\}$. This structure $M_{C,\lev}$ admits an iterated wreath product (in the sense of \cite[Section 6]{cameron1987some})
as an oligomorphic group of automorphisms which is a Jordan group.

Below, we sketch a proof of homogeneity of $M_{C,\lev}$. This appears not to be in the literature, but was shown for a richer structure, in which there is a total order on the domain compatible with $C$, in  \cite[Theorem 6.1, Theorem 6.2]{abyz}. The latter theorem shows that (in the setting of \cite{abyz}) one can put extra structure on the set of levels, and iterate the process. For convenience, for $a,b\in M_{C,\lev}$ we write $a\wedge b=q$ if $a\neq b$ and $q=\Min\{i: a(i)\neq b(i)\}$.

\begin{lem} The following hold for the structure $M_{C,\lev}$.
\begin{enumerate}
    \item 
The structure $M_{C,\lev}$ is homogeneous and so has quantifier elimination. 
\item If $(K,v)$ is any countable algebraically closed non-trivially-valued field, then there is a copy of $M_{C,\lev}$ with domain $K$ definable without parameters in  the valued field $(K,v)$.
\item $M_{C,\lev}$ is dp-minimal and so NIP. 
\end{enumerate}
\end{lem}

\begin{proof}
\begin{enumerate}
\item Let $\bar{a}=(a_1,\ldots,a_n)$, $\bar{b}=(b_1,\ldots,b_n)$ be tuples from $M_{C,\lev}$ such that the map $f$ with $f(a_i)=b_i$ for each $i$ is an isomorphism, and let $a\in M_{C,\lev}\setminus \{a_1,\ldots,a_n\}$. By back-and-forth, it suffices to find $b\in M_{C,\lev}$ such that we may extend $f$ to a partial isomorphism with domain $\{a_1,\ldots,a_n,a\}$ by putting $f(a)=b$.  Choose $j_0\in \{1,\ldots,n\}$ so as to maximise $a\wedge a_{j_0}$.

{\em Case 1.} There is $k_0\in \{1,\ldots,n\}\setminus\{j_0\}$ with $a\wedge a_{j_0}=a_{j_0}\wedge a_{k_0}$.  Now let
$$I:=\{i\in \{1,\ldots,n\}: a\wedge a_{j_0}=a_{j_0}\wedge a_i\}.$$
Put $q:= b_{j_0}\wedge b_{k_0}$. Choose $b\in M_{C,\lev}$ so that $b(r)=b_{j_0}(r)$ for $r<q$, and $b(q)\neq b_i(q)$ for each $i\in I$. 

{\em Case 2.} Not Case 1, but there are  distinct $k_0,l_0\in \{1,\ldots,n\}$ with $a\wedge a_{j_0}=a_{k_0}\wedge a_{l_0}$. Now let $q=b_{j_0}\wedge b_{l_0}$. Choose $b$ so that $b(r)=b_{j_0}(r)$ for $r<q$, and $b(q)\neq b_{j_0}(q)$.

{\em Case 3.} Not Cases 1 or 2. Now choose $q\in \mathbb{Q}$ so that $q<b_k\wedge b_l$ whenever $a\wedge a_{j_0}< a_k\wedge a_l$, and $q>b_k\wedge b_l$ whenever $a\wedge a_{j_0}> a_k\wedge a_l$. As in Case 2, choose $b$ so that $b(r)=b_{j_0}(r)$ for $r<q$, and $b(q)\neq b_{j_0}(q)$.

\item Define the relation $C$ on $K$ by putting $C(x;y,z)$ if and only if $v(x-y)\leq v(y-z)$, and $V(x,y;z,w)$ if and only if $x\neq y$ and $z\neq w$ and $v(x-y)<v(z-w)$. 

\item This follows from (ii); the dp-minimality of any model of the theory ACVF of algebraically closed valued fields is noted in \cite[Theorem A.11]{simon2015guide}.
\end{enumerate}
\end{proof}

\begin{figure}
	\begin{center}
			\includegraphics[scale=.5]{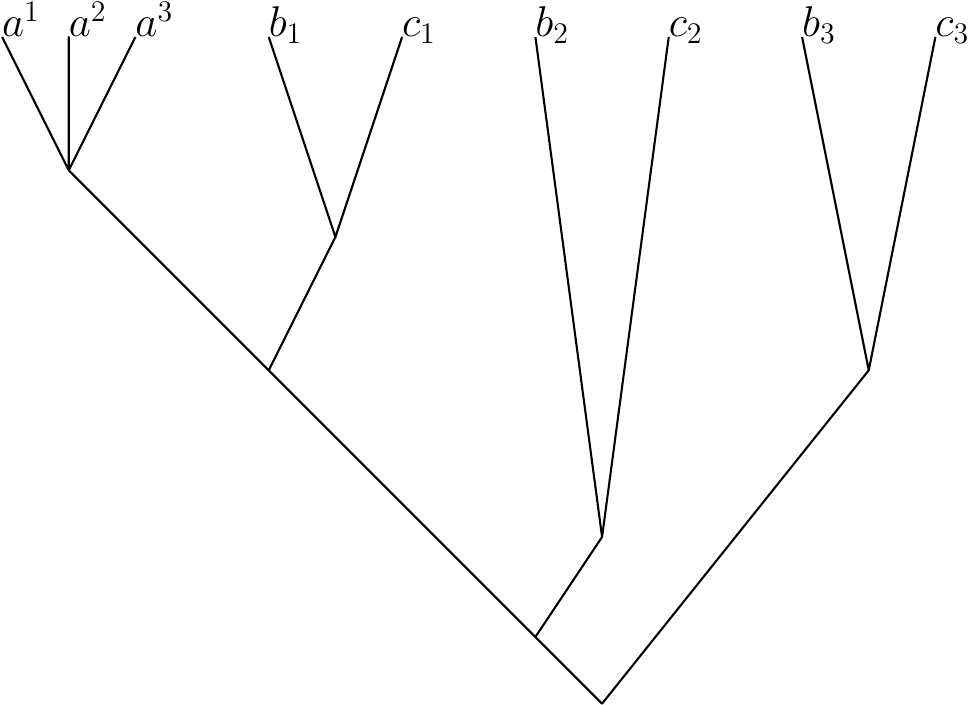}
		\end{center} 
  \caption{The structure $M_\sigma$, where $\sigma$ is the permutation of $\{1,2,3\}$ with $1 <_1 2 <_1 3$ and $2 <_2 3 <_2 1$. The order in which $\{b_i, c_i\}$ meets the leftmost branch defines $<_1$, while the height of $b_i \wedge c_i$ defines $<_2$.}
  \label{fig:Msigma}
\end{figure}

\begin{prop} \label{levelC} The structure $M_{C,\lev}$ satisfies the conclusions of Theorem~\ref{growth}, that is, 
\begin{enumerate}
   \item $(f_k(\Aut(M_{C,\lev}))\geq \lfloor k/3 \rfloor !$ for sufficiently large $k$, so in particular $(f_k(\Aut(M_{C,\lev}))$ has super-exponential growth rate.
    \item Age$(M_{C,\lev})$ is not wqo.
     \item $M_{C,\lev}$ is not monadically NIP.
\end{enumerate}
\end{prop}
\begin{proof}

The basic idea is to encode permutations (viewed as sets equipped with two linear orders $<_1, <_2$) of $\{1,\ldots,n\}$ into sets of size $2n+3$.  Suppose $n\geq 2$ and let $\sigma\in \Sym_n$. Pick distinct
$a^1, a^2, a^3,b_1,\ldots,b_n, c_1,\ldots,c_n\in M_{C,\lev}$ such that 
\begin{enumerate}
    \item $\neg C(a^i;a^j, a^k)$ for distinct $i,j,k \in \{1,2,3\}$
    \item $C(a^j;b_i,c_i)$ for every $j$ and $1\leq i\leq n$
    \item $a^k\wedge b_i<a^k\wedge b_j$ for every $k$ whenever $i<_1 j$
    \item $b_i \wedge c_i < b_j \wedge c_j$ whenever $i <_2 j$
\end{enumerate}
An example of the resulting structure $M_\sigma$ is depicted in  Figure \ref{fig:Msigma}. There is an equivalence relation $E$ on $\{b_i, c_i | i \in [n]\}$ whose classes are $\{\{b_i, c_i\} | i \in [n]\}$, which is definable from any $a_i$ by $E(x,y)$ if $x \wedge a_i = y \wedge a_i$. Given $E$-classes $E_j, E_k$ we may define $E_j <_1 E_k$ if $a_i \wedge E_j > a_i \wedge E_k$, where $a_i \wedge E_j$ is the element given by the meet of any $a_i$ with either element of $E_j$. We may define a second order $E_j <_2 E_k$ by $b_j \wedge c_j < b_k \wedge c_k$.

We claim that if $\sigma\not\hookrightarrow \tau$ then $M_\sigma \not\hookrightarrow M_\tau$. To see this, recall that embeddings must preserve the fact that an existential formula holds of a tuple. Now suppose  $f \colon M_\sigma \hookrightarrow M_\tau$ is an embedding. First observe that for any permutation $\pi$, set $\{a_1,a_2,a_3\}\in M_\pi$ is existentially definable by the formula $\phi(x)$ saying there exist $y,z$ such that $x \wedge y = x \wedge z = y \wedge z$. Note that fixing $a_i$, the definition of $E$ given above is quantifier-free, and so both $E$ and $\neg E$ are existentially definable. Similarly for the definitions of $<_1$ and $<_2$. Thus $f$ sends $\{a_1,a_2, a_3\} \subset M_\sigma$ to the corresponding 3-subset of $M_\tau$, and preserves $E$-classes and both linear orders on them. Thus $f(M_\sigma)$ encodes the permutation $\sigma$ as a substructure of $\tau$.

In particular, if $\sigma$ and $\tau$ are finite, then $\sigma \not\cong \tau$ implies $M_\sigma \not\cong M_\tau$. Part (1) of the Proposition now follows immediately, since $f_{2k+3}(\Aut(M_{C,\lev}))\geq k!$. 

Part (2) follows immediately from the existence of an infinite antichain of finite permutations under order-preserving embeddings. For this, see e.g. \cite[Figure 17]{vatter2015permutation}.

Finally, we show (3). Recall from \cite{braunfeld2024corrigenda} that a theory $T$ has {\em endless indiscernible triviality} if for every $M\models T$, every endpointless indiscernible sequence $I$ in $M$, and every tuple $B$ of parameters, if $I$ is indiscernible over each $b\in B$ then $I$ is indiscernible over $B$. It is shown in \cite[Theorem 1.2]{braunfeld2024corrigenda} that if $T$ is monadically NIP then $T$ has endless indiscernible triviality. In $M_{C, \lev}$ consider the infinite substructure $M_\sigma$ as above where $\sigma = (\mathbb{Z}, <_1, <_2)$ where both $<_1$ and $<_2$ agree with the standard order. Using quantifier-elimination in $M_{C, \lev}$, we see the sequence $I = (b_ic_i : i \in \mathbb{Z})$ is indiscernible over $A = \{a_1, a_2, a_3\}$. Consider a further pair of elements $b,c \in M_{C, \lev}$ corresponding to an $E$-class coming $<_1$-before every $E$-class of $I$  but defining a non-trivial $<_2$-cut in $I$. Then $I$ is still indiscernible over $Ab$ and over $Ac$ but not over $Abc$.
\end{proof}

{\em Proof of Theorem~\ref{growth}.} We give the argument for $\mathcal{M}=\mathcal{M}_D^2$. The proof is essentially the same for $\mathcal{M}=\mathcal{M}_B^2$.

We claim that the structure $M_{C,\lev}$ is definable on singletons in $\mathcal{M}$. 
For this, recall the paragraph after the proof of Proposition~\ref{groupM_B}. It was noted that if $[a]$ is a prenode of $\mathcal{M}_D^2$ at level $j_0$, then there is
a $G_{a/\sim_j}$-invariant sequence of equivalence relations $E_i$ on $[a]$ indexed by $\{i:i<j_0\}$ (which is order-isomorphic to $(\mathbb{Q},<)$) and for $i<k<j_0$ we have  $E_i\subset E_k$. This determines a $C$-relation on $[a]$ with $C(x;y,z)$ holding if for some $i$, some $E_i$-class contains $y,z$ and omits $x$. By \cite[Lemma 5.4]{almazaydeh2021jordan}, if $x,y\in [a]$ are distinct, then there is a greatest $i<j_0$ such that $\neg E_i xy$, and for such $i$ we have $F_i xy$. We shall put $\lev(x,y)=i$, and define $V$ on $[a]^4$ by putting $V(x,y;z,w)$ if $x\neq y$, $z\neq w$, and $\lev(x,y)\geq \lev(z,w)$.

Easily,  $C$ and $V$ are $a/\sim_j$-definable. Also, the structure $([a], C,V)$ is isomorphic to the structure $M_{C,\lev}$ of Proposition~\ref{levelC}. We do not give full details of this, but for example, in the proof of \cite[Lemma 5.7]{almazaydeh2021jordan} a bijection is constructed from $[a]$ (there denoted $[n]$) to a set $\Omega$ which is pretty clearly identifiable with the domain of $M_{C,\lev}$, and this bijection respects $C$ and $V$.

Now let $b,c\in X_j$ with $a,b,c$ inequivalent modulo $\sim_j$, so $j$ and hence $[a]$ are $abc$-definable. Then the structure $([a], C,V)$ is $abc$-definable, and we identify it with $M_{C,\lev}$ as above. Note that by the homogeneity of $\mathcal{M}$, these definitions can be taken quantifier-free, and so both they and their negations are preserved by passing to substructures.

Part (i) of  the theorem follows immediately from these observations and Proposition~\ref{levelC}, since we may name the parameters $a,b,c$ by unary predicates. 

Parts (ii) and (iii) follow by taking substructures of $([a], C,V)$ as in Proposition~\ref{levelC}. \qed

\medskip

We remark that while the first two points of Theorem \ref{growth} depend only on the definable relations of $\mathcal{M}$, and so are not dependent on the choice of language, the proof of the last point does use that we are working in the homogeneous language, because it is concerned with substructures of $\mathcal{M}$ and because embeddings need only preserve existential formulas. Nevertheless, we expect the third point to continue to hold regardless of the choice of language, in particular for $\mathcal{M} \in \{\mathcal{M}_B^0, \mathcal{M}_D^0\}$.

\section{Proof of Theorem~\ref{max}}
We show first that if $\mathcal{M}=\mathcal{M}_B^0$ and $G=\Aut(\mathcal{M})$, then $G$ is a maximal-closed subgroup of $\Sym(M)$. We omit the proof in the case when $H=\Aut(\mathcal{M}_D)$ -- the details are very similar. 

First observe that $G$ is a 3-homogeneous 2-primitive Jordan group. It follows from the description of primitive Jordan groups in \cite{adeleke1996classification} that if $G<H<\Sym(M)$ with $H$ closed, then one of the following holds.
\begin{enumerate}
    \item[(a)] $H$ preserves a linear separation relation.
    \item[(b)] $H$ preserves a $D$-relation.
    \item[(c)] $H$ preserves a Steiner system.
    \item[(d)] $H$ preserves a limit of Steiner systems.
    \item[(e)] $H$ preserves a limit of betweenness relations or $D$-relations.
    \end{enumerate}
To prove the theorem, we must eliminate each of these cases. Cases (a), (b), (c) are eliminated respectively in Lemmas 6.4, 6.6 and 6.5 of \cite{bhattmacph2006jordan}. 

To eliminate (e), suppose that $H$ preserves a limit of betweenness relations or $D$-relations. Then (by definition) $H$ is not 3-transitive. However, $G$ is 3-homogeneous and induces $C_2$ on each $3$-set. Since $C_2$ is maximal in $\Sym_3$, there is no such $H$. 

It remains to show that the group $H$ cannot preserve a limit of Steiner systems, and for this we argue as in 
\cite[Section 3]{bodirsky-macpherson}. We suppose for a contradiction that $H$ preserves a limit of Steiner $(n-1)$-systems, and follow the notation of Definition~\ref{limits-steiner}. As in \cite{bodirsky-macpherson}, we may suppose that the $X_i$ are all infinite. Also, arguing as above, as $G$ is 3-homogeneous and induces a maximal subgroup of $\Sym_3$ on 3-sets, we may suppose that $H$ is 3-transitive, so in particular $n\geq 3$. 

We argue as in \cite{bodirsky-macpherson}, but with the notion of {\em pre-branch}  in place of cone. We work in the case of $\mathcal{M}_B^0$. The main point is that pre-branches are a special class of Jordan sets for $G=\Aut(\mathcal{M}_B^0)$ with the property that if $U$ is a pre-branch, $A\subset U$  is finite, and $a\in U\setminus A$, then there is a pre-branch $U'\subset U$ with $a\in U'$ and $A\cap U'=\emptyset$. 

We claim that no set $X_j$ can contain a pre-branch.  To see this, suppose that $U$ is a pre-branch lying in $X_j$. Pick distinct $a_1,\ldots,a_{n-1}\in X_j$ and let $l$ be the unique Steiner line (of the Steiner system on $X_j$ containing these points). Let $b_{n-1}$ be a point of $X_j$ not on $l$, and $m$ be the Steiner line of $X_j$ containing $a_1,\ldots a_{n-2}, b_{n-1}$. As Steiner lines have more than $n-1$ points, there are $a_n\in l\setminus \{1_1,\ldots,a_{n-1}\}$ and $b_n\in m\setminus \{a_1,\ldots,a_{n-2},b_{n-1}\}$. So $l$ is the unique line containing  $a_1,\ldots,a_{n-3}, a_{n-1},a_n$  and $m$ is the unique line containing $a_1,\ldots,a_{n-3}, a_{n-1},a_n$, and $m$ is the unique line containing $a_1,\ldots,a_{n-3}, b_{n-1},b_n$. Let $A=\{a_1,\ldots,a_{n-3},a_{n-1},a_n,b_{n-1},b_n\}$. Since $A$ contains $n-1$ points of each of $l$ and $m$, any automorphism of the Steiner system on $X_j$ fixing $A$ fixes $l$ and $m$ setwise, so fixes the remaining point $a=a_{n-2}$ of their intersection. In particular, we have finite $A\subset X_j$ and $a\in X_j\setminus A$ such that $G_{(A)}\leq G_a$. By the last paragraph, we may choose a pre-branch $U'\subset U$ with $a\in U'$ and $A\cap U'=\emptyset$.
Now $G_{(M\setminus U')}$ contains an element $h$ with $a^h\neq a$ and $h\in H_{((M\setminus X_j)\cup A)} \leq H_{\{X_j\}}\cap H_{(A)}$. This is impossible as $H_{\{X_j\}}$ preserves the Steiner system on $X_j$.

\medskip

{\em Claim.} Any infinite subset $V$ of $M$ meets infinitely many disjoint pre-branches. 

\medskip

{\em Proof of Claim.}
{\em Case 1.} There is $j\in J$ such that
$V/\sim_j$ is infinite. Now let $k<j$. There is $a\in Z_k$ such that $V$ meets infinitely many branches of $Z_k$ at $a$. The corresponding pre-branches a $a$ are disjoint and satisfy the Claim.

{\em Case 2.} For each $j\in J$, $V/\sim_j$ is finite. In this case, there are $j_1>j_2>\ldots$ in $J$, indexed by $\omega$, such that for each $i\in \omega$, $|V/\sim_{j_{i+1}}|>|V/\sim_{j_{i}}|$. Pick $v_1\in Z_{j_1}$ such that $[v_1]\cap V\neq \emptyset$. There is $b_2\in Z_{j_2}$ such that $[v_1]$ is a union of pre-branches at $b_2$. Choose $W_1$ to be one of these pre-branches, meeting $V$. Now as $|V/\sim_{j_{2}}|>|V/\sim_{j_{1}}|$, we may pick $v_2\in Z_{j_2}$ disjoint from the above branches at $b_2$, such that
$[v_2]\cap V \neq \emptyset$. Again, there is $b_3\in Z_{j_3}$ such that $[v_2]$ is a union of pre-branches at $b_3$. We may pick $W_2$  to be one of these pre-branches, meeting $V$. Iterating this argument, we find a sequence $(W_i:i\in \omega)$ of disjoint pre-branches all meeting $V$. 

\medskip

Finally, pick a cofinal subset $\{i_n:n\in \omega\}$ of $J$. By the last paragraph, $X_{i_0}$ meets infinitely many disjoint pre-branches $\{U_n:n\in \omega\}$, so for each $n\in \omega$
there is $x_n\in X_{i_0}\cap U_n$. By the previous paragraph there is also $y_n\in U_n\setminus X_{i_n}$, for each $n$. Using that pre-branches are Jordan sets, there is $g\in G$ with $x_n^g=y_n$ for each $n$. Thus, for each $n$ we have $X_{i_0}^g\cap (M\setminus X_{i_n})\neq \emptyset$, so for each $j\geq i_0$ and each 
$k\in J$ we have $X_j^g\cap (M\setminus X_k)\neq \emptyset$. This contradicts condition (iv) in Definition~\ref{limits-steiner}. \qed
\bibliographystyle{plain}
\bibliography{references}
\end{document}